\documentclass[11pt]{amsart}

\usepackage{amsmath,amssymb,amsthm,hyperref,geometry}
\usepackage{graphicx}
\usepackage{geometry}
\usepackage{tikz}  
 \newtheorem{theorem}{Theorem}

 \newtheorem{lemma}[theorem]{Lemma}
 \newtheorem{proposition}[theorem]{Proposition}
 
 {\theoremstyle{definition}
 \newtheorem{definition}{Definition}
 \newtheorem{remark}{Remark}
 
 \newtheorem{conjecture}{Conjecture}
 }
      
 \newcommand\zu{[0,1]}
\newcommand{\N}{\ensuremath{\mathbb N}} 
\newcommand{\R}{\ensuremath{\mathbb R}} 
\newcommand{\Z}{\ensuremath{\mathbb Z}} 
\newcommand\CC{\mathcal{CC}^d}
\newcommand{\ep}{\varepsilon}
\newcommand{\suppmu}{\mbox{Supp}(\mu)}
\newcommand{\supp}{\mbox{Supp}}
\newcommand\mk{\medskip}

\newcommand\ml{multifractal }
\newcommand \omegat{ \omega^{\mu}}
\newcommand\si{\sigma}
\newcommand\btilde{\widetilde{B}}

\newcommand{\la}{\lambda}

\title{A survey on prescription of  multifractal behavior}
\author{St\'ephane Seuret }
\address{St\'ephane Seuret, Universit\'e Paris-Est, LAMA (UMR 8050),  UPEMLV, UPEC, CNRS, F-94010, Cr\'eteil, France}
\email{seuret@u-pec.fr }
\thanks{
Research   partly supported by the grant ANR MULTIFRACS.
} 

\date{\today}
\begin{document}

\maketitle

\medskip

\begin{abstract}
Multifractal behavior has been identified and  mathematically established  for large classes of functions, stochastic processes and measures.  Multifractality has also been observed on many data coming from Geophysics, turbulence, Physics, Biology, to name a few. Developing mathematical models whose scaling and multifractal properties  fit those measured on data is thus an important issue. This raises several still unsolved theoretical questions about the prescription of multifractality (i.e. how to build mathematical models with a singularity spectrum known in advance), typical behavior in function spaces, and existence of solutions to PDEs or SPDEs with possible multifractal behavior. In this survey, we gather some of the latest results in this area.
\end{abstract}
 
 \
 
\begin{center}
{\em Dedicated to   Alain Arn\'eodo,\\ pioneer in the development of wavelet tools for data analysis. 
}
\end{center}
 
 \
 
\section{Multifractality between pure and applied mathematics}\label{introduction}

The notion of multifractal functions and measures can be traced back to the interest of physicists in  the H\"older singularities structure in fully developed turbulence, which is described in terms of large deviations for the distribution at small scales of Mandelbrot random multiplicative cascades in \cite{M2}, and in a geometric setting in the version of the so-called multifractal formalism for functions proposed by Frisch and Parisi \cite{FrischParisi}, see Section \ref{sec-fp}. Another source leading to   multifractal ideas  is provided     by the works of Henschel \& Procaccia \cite{Hentschel} and Halsey $\&$ al. \cite{Halsey}. Since then, multifractal analysis was further developed  in dynamical systems theory and geometric measure theory, and has become a standard tool to describe the fine geometric structure of objects possessing nice invariance properties, such as   self-similar and self-affine   measures and functions, many classes of  stochastic processes such as    L\'evy processes and more general Markov processes, as well as random measures emerging from multiplicative chaos theory. 

\medskip

Let us  recall  the notion of singularity spectrum of a function, leading to multifractals.

Let $d\ge 1$ be an integer. Given a real
function $f \in L^\infty_{loc}(\R^d)$ and $x_0 \in \R^d$, $f$ is said to
belong to ${\mathcal C}^{H}(x_0)$, for some $H\geq 0$, if
there exists a polynomial $P$ of degree at most $\lfloor H \rfloor$
and a constant $C>0$ such that 
$$
\mbox{for $x$ close  to $x_0$, } \ \ \   |f(x)-P(x-x_0)|\leq C | x-x_0|^{H}.
$$

\begin{definition}
 The {\it pointwise
  H\"older exponent}  of $f\in L^\infty_{loc}(\R^d)$  at $x_0$ is 
$$h_f(x_0)=\sup\left \{ H \geq 0:~f \in {\mathcal C}^{H}(x_0)\right \},$$
and  $f$ is said to have a H\"older singularity of order $h_f(x_0)$ at $x_0$. 

The {\it singularity spectrum} $D_f$ of $f$  is the
map: $$D_f: H \in[0,\infty] \longmapsto \dim 
\, E_f(H), \  \mbox{ where } E_f(H):= \{x_0\in \R^d: h_F(x_0)=H\}.$$
\end{definition}

The notation $\dim$ stands for the Hausdorff dimension, and by convention $\dim \emptyset =-\infty$.

The multifractal spectrum $D_f$ encapsulates key information on  a given function $f$, in particular it  carries a description of the distribution of the singularities of  $f$. But the computation of $D_f$ often raises deep mathematical questions (for instance, it took almost 130 years to find the multifractal spectrum of the famous Riemann series $\displaystyle  \sum_{n=1}^{+\infty}\frac{\sin (n^2\pi x)}{n^2}$), and in most cases  the exact value of $D_f$ happens to be not directly accessible, neither theoretically nor numerically.

Fortunately, the notion of multifractal formalism furnishes a  clever way to circumvent this difficulty and to compute the explicit value of the spectrum of large classes of measures and  functions. Also, multifractal formalism provides ideas to develop   numerical algorithms able to estimate $D_f$ on real-life data. The main idea is that for very large classes of functions $f$ (and also for other mathematical objects like  measures, stochastic processes - such examples will be given in this paper), $D_f$ is equal to the Legendre transform of the so-called $L^q$-spectrum $\tau_f$ of $f$: this $L^q$-spectrum is computed directly using the values of $f$, and is numerically accessible. When  these two quantities ($D_f$ and the Legendre transform of $\tau_f$) coincide, it is said that $f$ satisfies the multifractal formalism.  Examples of $L^q$-spectra  for functions (and measures) based  on increments, wavelet coefficients or wavelet leaders, are given in the upcoming sections  (see \eqref{defzeta}, \eqref{def-tau}, \eqref{def-eta} or  \eqref{def-L}). The intuition that a multifractal formalism should hold is due to U. Frisch and G. Parisi, we refer the reader to  Section  \ref{sec-fp} for an account on the ideas leading to this formula.

\begin{figure}
\includegraphics[scale=1.3]{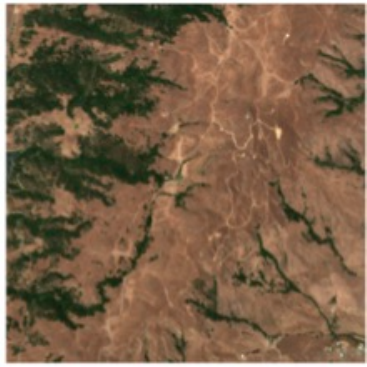}  \includegraphics[scale=1.2]{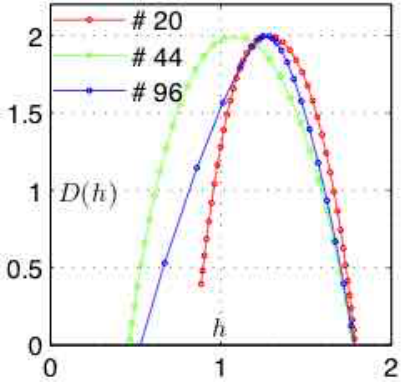}
\caption{Image and estimated multifractal spectrum of different color le\-vels of a satellite  image. {\em Courtesy H. Wendt}.} 
\label{fig_examples}
\end{figure}

The multifractal formalism, and its validity for many mathematical models,  explains the success of the multifractal approach    used as classification tool in signal and image processing. Indeed,    algorithms have been developed (mainly based on wavelet theory, see \cite{Muzy-arneodo-1} for the original WTMM method and more recently  \cite{AJW-1} for a mathematical study of the wavelet leaders algorithm and the latest developments and algorithms based on wavelet leaders)  to estimate numerically $L^q$-scaling functions, the stability and efficiency of these algorithms being mathematically grounded. Using these algorithms, it is now established that  many data coming from Geophysics, turbulence, Physics, Biology,   exhibit non-linear $L^q$-scaling functions, which for a given function $f$ is interpreted thanks to the Frisch-Parisi heuristics as a non-trivial singularity spectrum $D_f$  of  $f$. Examples of data and estimated singularity spectra are plotted in Figure \ref {fig_examples} and \ref{fig_examples2}.

\medskip

Resuming the above, we have on one side many mathematical objects $f$ with non-linear $L^q$-scaling functions and a non-trivial singularity spectrum $D_f$, and on the other side an impressive quantity of signals, images and  multivariate, multi-dimensional data whose estimated $L^q$-spectra and singularity spectra are non-trivial. It is worth asking which mathematical objects are indeed the most relevant to model the observed data, and how to create models with any reasonable  \ml behavior.

This general problematics can be understood in various ways, and raises several theoretical questions, most of them still  being   open:
\begin{enumerate}
\item
What are the  mappings $\sigma: \R^+\to  [0,d]\cup\{-\infty\}$ that are admissible to be a multifractal spectrum, i.e. there exists a function $f:\R^d \to \R $  such that $D_f = \sigma$?

\item
 What are the  mappings $\sigma: \R^+\to  [0,d]\cup\{-\infty\}$ that are admissible to be a {\em homogeneous} multifractal spectrum, i.e.  there exists a function $f:\R^d \to \R$  such for every  cube $I\subset \R^d$ with non-empty interior,  $D_{f_I} = \sigma$ where $f_I$ stands for the restriction of $f$ on $I$?

\item
Given an admissible (homogeneous or not) singularity spectrum $\sigma: \R^+\to  [0,d]\cup\{-\infty\}$, is there a functional space in which Baire typical functions have $\sigma$ as singularity spectrum? Do typical functions satisfy a \ml formalism?

\item
Given an admissible  (homogeneous or not) singularity spectrum $\sigma: \R^+\to  [0,d]\cup\{-\infty\}$, is there a differential equation, a PDE or a stochastic (P)DE whose solution has $\sigma$ as singularity spectrum? 
\end{enumerate} 

\begin{figure}
 \includegraphics[scale=1.5]{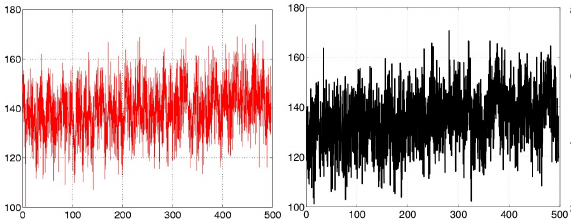} \includegraphics[scale=1.5]{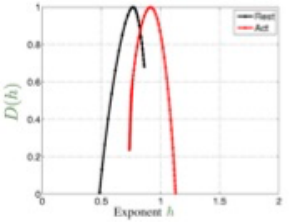}
\caption{Two FMRI signals  of a resting (in black) and acting (in red) patient. Comparison between  their estimated \ml spectrum. {\em Courtesy H. Wendt.}} 
\label{fig_examples2}
\end{figure}

These problems have their counterpart in terms of $L^q$-spectra: replacing everywhere  $\sigma: \R^+\to  [0,d]\cup\{-\infty\}$  by $\tau:\R\to\R$, one may ask for the admissible $\tau$ that can be the $L^q$-spectrum of a function (homogeneous or not), and if such an $L^q$-spectrum is typical in some functional space. 

The same questions arise when considering probability measures instead of functions. The main difference with the function setting is that there are additional constraints when dealing with measures, see Sections \ref{sec-expo} and   \ref{sec-prescrmu}.

\medskip

Although the tools used in the two contexts (functions and measures) are of different nature, a connection between the two situations is provided by the following theorem from \cite{BSFmu}, based on wavelet analysis.

\begin{theorem}
\label{th-Fmu}
Let $\mu$ be a probability measure on $\R^d$ such that there exist  $\alpha,C>0$  satisfying that for every $x\in \R^d$ and $0\leq r \leq 1$, $\mu(B(x,r)) \leq C r^\alpha$.

Consider the function $F_\mu:\R^d\to \R $ whose wavelet coefficients are given by $d_\lambda = \mu(\lambda)$ for every dyadic cube $\lambda\in \Lambda$ (see Section \ref{sec-prescrfunc} for   definitions).

Then the multifractal spectra of $\mu$ and $F_\mu$ coincide.
\end{theorem}

\medskip

Our purpose here is to provide a survey on   recent results and  on some open problems related to these various research directions, which combine many ideas coming from (and having applications to) geometric measure theory, functional and harmonic analysis, and real analysis, as well as ergodic theory and dynamical systems.

\section{Prescription of   exponents and local dimensions}
 \label{sec-expo}

For a given mapping $f:\R^d\to \R$ belonging to $L^\infty_{loc}(\R^d)$, its associated  pointwise H\"older exponent mapping $h_f:x\mapsto h_f(x)$ may be very erratic, changing violently from one point to the other.  Nevertheless $h_f$ (viewed as a function) is quite well understood, as confirmed by  the following theorem by S. Jaffard which provides  a full characterization of $h_f$ \cite{JAFFNOTE,jaffard-prescribed}. Recall that $C^{\log}(\R^d)$ is the space of those functions $f:\R^d\to\R$ satisfying that there exists $C>0$ such that for every $x,y\in \R^d$ with $|x-y|\leq 1/2$, $|f(x)-f(y)| \leq C |\log|x-y||^{-1}$.

\begin{theorem}
\label{th_jaf1}
When $f\in C^{\log}(\R^d)$,  the mapping $h_f$ is a liminf of a sequence of continuous functions. 

Conversely, let $H:\R^d \to \R^+\cup\{+\infty\}$ be a  liminf of a sequence of continuous functions. There exists a function $f:\R^d\to \R$, 
$f\in C^{\log}(\R^d)$, such that for every $x\in \R^d$, $h_f(x)=H(x)$.
\end{theorem}

Let us also mention that in \cite{AJT-prescription} the authors build a continuous nowhere differentiable stochastic process $(M_x)_{x\geq 0}$ whose  pointwise H\"older exponents have the most general form, i.e. the mapping $x\mapsto h_M(x) \in (0,1)$ can be any  liminf of a sequence of continuous functions.

\medskip

It is a natural question to investigate the same issues for local dimensions for measures. 

\begin{definition}
Let $\mathcal{M}(K)$ be the set of Borel probability measures on a Borel set $K\subset \R^d$.

For $\mu \in \mathcal{M}(\R^d)$, the support of $\mu$ is the set 
$$\supp(\mu)=\{x\in \R^d: \mu(B(x,r))>0 \mbox{ for every } r>0\}.$$

The (lower) local dimension of $\mu$ at $x\in\supp(\mu)$ is
\begin{equation}
\label{defhmu}
h_\mu(x) = \liminf_{r\to 0^+} \frac{\log \mu(B(x,r))}{\log r}
\end{equation}
and the singularity spectrum of $\mu$ is defined for $H\in \R\cup\{+\infty\}$ by 
$$D_\mu(H) = \dim E_\mu(H) \  \ \ \mbox{ where } E_\mu(H) = \{ x\in \suppmu : h_\mu(x) = H\}.$$
\end{definition}

It is common  (and in many situations, relevant and important) to look at points $x$ at which \eqref{defhmu}
turns  out to be a limit (and not only a liminf). Nevertheless, in this article   only   lower local dimensions  are considered (we will forget the term "lower" in the following), since we are interested in quantities defined for all $x\in \suppmu$.

\begin{definition}
A function $f$ (resp. a measure $\mu$) on $\R^d$ is called homogeneous (in short: HM) if the restriction of $f$ (resp. $\mu$) on any finite subcube $I\subset \R^d$  has the same singularity spectrum as $f$ (resp. $\mu$).

The same definition applies to a function or measure when $\R^d$ is replaced by $\zu^d$.
\end{definition}

One could expect that an analog of Theorem \ref{th_jaf1} should hold for local dimensions of measures. Unfortunately, the situation is not as clear, as proved by the next lemma  \cite{BuS3}.

\begin{lemma}
Let $\mu \in \mathcal{M}(\R^d)$ with a support   containing a cube $U\subset \R^d$. If the mapping $x\mapsto h_\mu(x)$ is continuous on  $U$, then $h_\mu$ is locally constant and equal to $d$ on $U$.
\end{lemma}

Last lemma leads to the two following open problems: What  are the admissible mappings $H:\R^d\to\R^+$ satisfying $H=h_\mu$ for some probability measure $\mu$? Given an admissible mapping $H$, can one explicitly build a measure $\mu \in\mathcal{M}(\R^d)$ such that $h_\mu=H$? 

\medskip

Even if all these  questions are mathematically   relevant and raise delicate questions (in geometric measure theory for instance), in many situations it is even more important  to construct functions with prescribed singularity spectrum. This is the case in particular when trying to model real-life data, for which essentially only global quantities (like the $L^q$-spectrum) are accessible.  

\section{Prescription of multifractal behavior}

As expected, the prescription of singularity spectrum for functions or measures is more involved than that of    exponents.
Indeed, there is no obvious characterization for the admissible singularity spectrum for functions. Yet, using   wavelet techniques, S. Jaffard was able to prove the following theorem \cite{Jaff-prescription}. Let 
$$\mathcal{R} = \left \{\si : \R^+ \to [0,d]\cup\{-\infty\} :\begin{cases}  \exists \mbox{ bounded  interval $I\subset \R^+$ and $\alpha\in [0,d]$}\\ \mbox{ such that $ \si= \alpha{\bf 1\!\!\!1}_{I} + (-\infty){\bf 1\!\!\!1}_{\R^+\setminus I}$}\end{cases} \right\}.
$$

\begin{theorem}
\label{th-ps1}
 Let $\sigma:\R^+\to [0,d]\cup\{-\infty\}$ be the supremum of a countable sequence of functions $(\si_n)_{n\geq 1}\in \mathcal{R}$.
Then there exists a continuous function $f:\R^d\to \R$ such that $D_f=\sigma$.
\end{theorem}
 Although probably not optimal, this theorem already covers a  large class of  singularity spectra, certainly sufficient to mimic precisely all the singularity spectra that can be estimated on real data.
 
 In particular, any concave mapping $\sigma :\R^+\to [0,d]\cup\{-\infty\}$ can be written as $\sup_{n\in \N} \si_n$ for some well chosen functions $\si_n\in \mathcal{R}$, hence it is possible to build   a  function $f:\R^d\to \R$ such that $D_f=\sigma$.   
 
 \medskip
 
 The same questions were addressed for measures first in \cite{BuS3} and then in \cite{Barralinverse}.  The admissible singularity spectra for measures are not characterized either, but when compared to spectra of functions,  there are additional constraints: if $d_\mu=\sigma$ for some $\mu \in \mathcal{M}(\R^d)$, then $\sigma(h)\leq \min(h,d)$ (see \cite{BroMichPey,Olsen}).

 Another surprising  constraint obtained in \cite{BuS3} is that the support of the singularity spectrum of a 1-dimensional HM measure contains  an interval. We call $\supp(\si)$ the support of a function $\si:\R^d\to \R$, and by abuse of notation, if $\si:\R\to \R^+\cup\{-\infty\} $, $\supp(\sigma)=\{ H:\si(H)\geq 0\}$.
 
  \begin{proposition}
 For any non-atomic HM probability measure $\mu \in\mathcal{M}(\R)$, Supp($D_\mu) \cap [0,1]$ is necessarily an interval of the form $[\alpha, 1]$, where $0\leq \alpha\leq 1$.
 \end{proposition}
 This proposition leads to the following notation: for $\sigma:\R^+\to [0,1]\cup\{-\infty\}$,  consider the mapping  
 $$\sigma^{\dagger}(H) =\max \big(\sigma(H), 0\cdot {\bf1\!\!\!1}_{[\inf(\mbox{\tiny \supp}(\sigma)),\sup(\mbox{\tiny \supp}(\sigma))]}(H)\big).$$
  Essentially, $\sigma^{\dagger}$ fills the gaps in the support of $\sigma$ by replacing the value $-\infty$ by $0$.

 \medskip
 
 The result concerning the prescription of singularity spectrum of measures obtained in \cite{BuS3}  is the following.

\begin{theorem}
\label{th-ps2}
 Let $\sigma:\R^+\to [0,1]\cup\{-\infty\}$ be the supremum of a countable sequence of functions $(\si _n)_{n\geq 1}\in \mathcal{R}$ satisfying in addition that for every $n\geq 1$,  calling $I_n$ the interval on which $\si_n$ is not $-\infty$,
 \begin{itemize}
 \item
 $I_n\subset [0,1]$,
 \item
  $I_n$ is closed,
 \item
 $\si_n(x)\leq x$ for $x\in I_n$.
 \end{itemize}
Then:\begin{enumerate}
\item
There exists $\mu \in \mathcal{M}(\R )$ such that $D_\mu=\sigma$.

\item There exists a HM measure  $\mu \in \mathcal{M}(\R)$ with support equal to $\zu$  such that $D_\mu= \sigma^{\dagger}$, and $D_\mu(1)=1$.
\end{enumerate}
\end{theorem}

Observe that although the class of singularity spectra obtained here is quite large, only local dimensions less than $1$ are dealt with, and only the   one-dimensional case is covered.
 
 Theorem \ref{th-ps2} is completed by the result by Barral \cite{Barralinverse}.
\begin{theorem}
\label{th-ps3}
 Let $\sigma:\R^+\to [0,d]\cup\{-\infty\}$ be an upper semi-continuous function with support included in  $[\alpha,\beta]$ for some $0<\alpha<\beta<+\infty$, satisfying $\sigma(h)\leq h $ for every $h\in [\alpha,\beta]$, and  such that $\sigma(h)=h$ for some $h$. Then there exists $\mu \in \mathcal{M}(\R^d )$ such that $D_\mu=\sigma$.
\end{theorem}

 In the last theorem, Barral was also able to build measures that were "homogeneous" in the sense that the restriction of $\mu$ to any bounded cube $I\subset \R^d$ such that  $\mu(I)\neq 0$ has the same singularity spectrum as $\mu$ itself. A comparison between Theorems \ref{th-ps2} and \ref{th-ps3} yields that (at least) in dimension 1, the measures constructed by Barral are necessarily not supported by a full  interval (their support is a Cantor-like set), otherwise $\sigma$ should be replaced by $\sigma^{\dagger}$.
 
 \medskip

Theorems \ref{th-ps1}, \ref{th-ps2} and \ref{th-ps3} are not entirely satisfying. Indeed, 
\begin{itemize}
\item
 the construction  used in Theorem \ref{th-ps1} does not guarantee that the corresponding spectrum is homogeneous. Homogeneous spectra are yet very common (for instance, trajectories of stationary processes usually exhibit homogeneous spectra).
 \item
in the three previous theorems, even  if the prescribed spectrum is concave, the corresponding function  or measure {\em a priori}  does not satisfy a multifractal formalism.
\item
the functions and measures built along the proofs of Theorems \ref{th-ps1} and \ref{th-ps2} are  not "typical" in any sense, and may essentially appear, from the modeling standpoint, as mathematical extreme toy examples.
 
 \end{itemize}
 
 These issues will be addressed in the next sections.

\section{Prescription of multifractal formalisms}
\label{sec-fp}

  Let us very quickly recall the intuition by  Frisch \& Parisi \cite{FrischParisi}, who studied the velocity $v$ of a turbulent fluid in a bounded domain $\Omega\subset \R^3$. More precisely, inspired by the seminal works by Kolmogorov on turbulent fluids and the study of the local fluctuations of their velocity, Frisch  and  Parisi  were interested in the 
 moments of the increments of   $v$ defined by
 \begin{equation}
 \label{defzeta}
 \mbox{for every $q\in \R$, } \ \ S_v(q,l) = \int_{\Omega} | v(x+l)-v(x) |^q dx.
 \end{equation} 
For real data, $q$ being fixed, it has been observed  that   when  $| l | $ becomes small, $S_v(q,l) $ obeys a scaling law:
$$S_v(q,l) \sim |l |^{\zeta_v (q)} \ \ \mbox{for some exponent $\zeta_v (q)\in\R$.}$$

The mapping   {$q\mapsto \zeta_v (q)$} is called the   {scaling function} of the velocity of the fluid. It can be  seen that if $v$ is modeled at small scales by a fractional Brownian motion of index $H_0$  (as did Kolmogorov for instance), then $\zeta_v(q)$ is linear with  slope $H_0$. However, in the 1980's, numerical experiments for the velocity show that  $\zeta_v(q)$ is   {non-}linear and   {concave.} The seminal idea by Frisch and Parisi consists in  interpreting this non-linearity in terms of multifractality of $v$, via  the 
  following heuristic argument. 
  
  Replacing   Hausdorff by   box dimension, and making all kind of rough approximations (i.e. assuming that limits exist, etc),  for all points $x\in\R^3$ at which $h_v(x)= H $, one has
$| v(x+l)-v(x) | \sim |l|^ H $ for small $l$.  Since $\dim  E_v( H ) = D_v( H )$, there should exist approximately $|l|^{-D_v( H )}$ cubes of size $l$ in the domain $\Omega$ containing points  $x$  which are singularities of order $H$ for the velocity $v$.  All these intuitions lead to the estimates
$$S(q,l) = \int_{\Omega} | v(x+l)-v(x) |^q dx  \sim \sum_H   | l |^{q H } |l|^{-D_v( H )}|l|^3    \sim 
\sum_H   | l |^{q H -D_v( H )+3}  .  $$
 
When $|l|\to 0$, the greatest  contribution is obtained for the smallest exponent:
$$\zeta_v(q)=\inf_H   \ (q H -D_v( H )+3).$$
 The corresponding mapping $q\mapsto \zeta_v(q)$ is called the $L^q$-spectrum or the scaling function of $v$ -  soon we will  see more relevant formulas for $\zeta_v (q)$ and how to define it for measures.
 
By inverse Legendre transform, one deduces that
$$ {D_v( H )= \inf_{q\in\R}  \ (q H -\zeta_v(q)+3)}$$
which justifies that   $D_v$ has a concave shape.

\mk

It is striking that despite   the series of crude approximations, this intuition has proved to hold true in many (if not most of) situations, after some renormalization and suitable choices for the scaling functions.

 \begin{definition}
 \label{def-formalism}
 We call  {multifractal formalism} any formula relating the singularity spectrum of a function (or a measure) to a scaling function via a Legendre transform. 
 \end{definition}

For almost 30 years now, many efforts  have been  made   to prove the validity of  multifractal formalism(s) in various functional spaces, for many mathematical objects  (self-similar or self-affine functions and measures) including  random processes (Mandelbrot cascades, Gaussian multiplicative chaos, L\'evy processes). This line of research was constantly  followed and fostered by applications which gave mathematicians lots of signals and physical phenomena to study and work on, see Figures \ref{fig_examples} and \ref{fig_examples2}. In particular, stable algorithms to estimate $L^q$-spectra of data have been developed, furnishing to the scientific community many robustly analyzed sets of data \cite{AJW-1}. 
 
 \mk
 
A remaining question though lies in the existence of a functional setting in which a given multifractal behavior would be "generic". This is known after \cite{JAFF_FRISCH} as the {\em Frisch-Parisi conjecture}, which can   be formulated as follows:

\begin{conjecture}
\label{conjfp}
Given any admissible concave mapping $\sigma:\R^+\to [0,d]\cup\{-\infty\}$, is there a functional space in which typical functions have $\sigma$ as singularity spectrum and satisfy a \ml formalism?
\end{conjecture}

Notice that ideas leading to a \ml formalism can also be found in thermodynamics (see  \cite{Halsey,Hentschel} and the large literature around thermodynamical formalism). This outlines the universality of the approach consisting in  describing local fluctuations via the (Legendre transform  of) global statistical quantities computed directly on the object (function, measure, random process) under consideration.

\medskip
 
 From now on,  and without loss of generality, we restrict our statements to measures and functions  supported in the cube $\zu^d$.

\subsection{Prescription of multifractal formalism for measures}
\label{sec-prescrmu}

In case of measures $\mu\in\mathcal{M}(\zu^d)$, the formula for the $L^q$-spectrum is quite standard and given by
\begin{equation}
\label{def-tau}
\tau_\mu(q) = \liminf_{j\to +\infty} \frac{1}{-j} \log_2 \sum_{\lambda \in {\mathcal{D}}_j: \, \mu(\lambda)\neq 0} \mu(\lambda)^q,
\end{equation}
where $\mathcal{D}_j$ stands for the set of dyadic cubes $\lambda_{j,k}=2^{-j} k+[0,2^{-j}]^d$, $k\in\Z^d$,  of generation $j\in \Z$  (i.e. dyadic cubes with side-length equal to $2^{-j}$). It is easily seen that $\tau_\mu$ is always concave, non-decreasing,   and that $ - d \leq \tau_\mu(0^+)\leq \tau_\mu(1) = 0$. In addition, the support of $\tau_\mu$  is equal to $\R$ when 
$\limsup_{r\to 0^+} \frac{\log(\inf \{\mu(B(x, r)) : \  x \in {\tiny\suppmu\}})}{\log r} <+\infty$, and it is $[0,+\infty)$ when the same quantity is infinite \cite{Barralinverse}. 

Recall that the Legendre transform of a mapping $\tau:\R\to\R$ (used in the previous section) is defined for $H\geq 0$ as
$$\tau^*(H):= \inf_{q\in \R} (qH -\tau(q)).$$

Barral solved in \cite{Barralinverse} the following inverse problem.
\begin{theorem}
Let $\tau:\R\to \R$ be concave, non-decreasing, with   $ -d \leq \tau(0^+)\leq \tau(1) =0$.    There exists a probability  measure $\mu\in \mathcal{M}(\zu^d)$  compactly supported, such that $\tau_\mu=\tau$ and $\mu$ satisfies the multifractal formalism, i.e. $D_\mu=\tau^*  $.
\end{theorem}

See Figure \ref{figureFP} for an illustration.


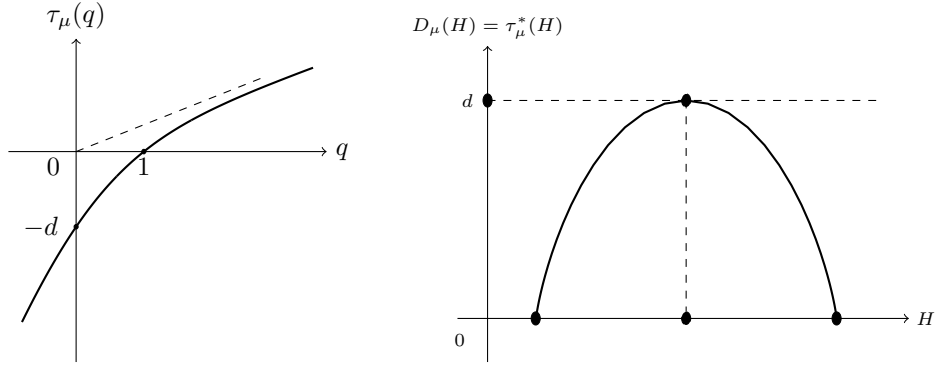
\begin{figure}  \begin{center}\begin{tikzpicture}[xscale=0.9,yscale=1.0]
{\small
\draw [->] (0,-2.8) -- (0,1.5) [radius=0.006] node [above] {$\tau_\mu(q)$};
\draw [->] (-1.,0) -- (3.7,0) node [right] {$q$};
 \draw [thick, domain=1.7:3.5, color=black]  plot ({\x},  {-ln(exp(\x*ln(1/5)) +exp(\x*ln(0.8)))/(ln(2)});
 \draw [thick, domain=0.6:1.7, color=black]  plot ({\x},  {-ln(exp(\x*ln(1/5)) +exp(\x*ln(0.8)))/(ln(2)});
\draw [thick, domain=-0.8:0.6, color=black]  plot ({\x},  {-ln(exp(\x*ln(1/5)) +exp(\x*ln(0.8)))/(ln(2)});
  \draw[dashed] (0,0) -- (2.8,1);
\draw [fill] (-0.1,-0.20)   node [left] {$0$}; 
\draw [fill] (-0,-1) circle [radius=0.03]  node [left] {$-d$ \ }; 

\draw [fill] (1,-0) circle [radius=0.03] [fill] (1,-0.2) node [ ] {$1 $}; }
\end{tikzpicture} \hskip .5cm
 \begin{tikzpicture}[xscale=2.0,yscale=2.9]
    {\tiny
\draw [->] (0,-0.2) -- (0,1.25) [radius=0.006] node [above] {$  {{D}}_\mu(H)=\tau_\mu^*(H)$};
\draw [->] (-0.2,0) -- (2.8,0) node [right] {$H$};
\draw [thick, domain=0:5]  plot ({-(exp(\x*ln(1/5))*ln(0.2)+exp(\x*ln(0.8))*ln(0.8))/(ln(2)*(exp(\x*ln(1/5))+exp(\x*ln(0.8)) ) )} , {-\x*( exp(\x*ln(1/5))*ln(0.2)+exp(\x*ln(0.8))*ln(0.8))/(ln(2)*(exp(\x*ln(1/5))+exp(\x*ln(0.8))))+ ln((exp(\x*ln(1/5))+exp(\x*ln(0.8))))/ln(2)});
\draw [thick, domain=0:5]  plot ({-( ln(0.2)+ ln(0.8))/(ln(2)) +(exp(\x*ln(1/5))*ln(0.2)+exp(\x*ln(0.8))*ln(0.8))/(ln(2)*(exp(\x*ln(1/5))+exp(\x*ln(0.8)) ) )} , {-\x*( exp(\x*ln(1/5))*ln(0.2)+exp(\x*ln(0.8))*ln(0.8))/(ln(2)*(exp(\x*ln(1/5))+exp(\x*ln(0.8))))+ ln((exp(\x*ln(1/5))+exp(\x*ln(0.8))))/ln(2)});
  \draw[dashed] (0,1) -- (2.6,1);
\draw [fill] (-0.1,-0.10)   node [left] {$0$}; 
\draw [fill] (0,1) circle [radius=0.03] node [left] {$d \ $}; 
\draw  [fill] (0.32,0) circle [radius=0.03] node [below] { };
\draw  [fill] (2.32,0) circle [radius=0.03] node [below] {  };
\draw  [fill] (1.32,1) circle [radius=0.03]  [dashed]   (1.32,1) -- (1.32,0)  [fill] (1.32,0) circle [radius=0.03]  node [below] { };
}
\end{tikzpicture}
\end{center}
\caption{{\bf Left:}  $L^q$-spectrum  of a  measure $\mu$ on $[0,1]$. {\bf Right:} The corresponding singularity spectrum of $\mu$ when it  satisfies a multifractal formalism.}
\label{figureFP}
\end{figure}

The  drawback of this first important step is that  the  measure constructed by Barral in \cite{Barralinverse}  has again a Cantor-like set as support (so it is not fully supported on $\zu^d$), hence is not suitable to model any  real-life signal supported by, say, an interval.
The result is reinforced in the upcoming paper \cite{BS-FP}, in which we build fully supported measures satisfying a prescribed multifractal formalism, which in addition are almost-doubling in the following sense. 

A Borel set function is a mapping $\mu$ associating with every Borel set $B\subset \zu^d$ a positive real number $\mu(B) \in [0,+\infty]$. A Borel set function  $\mu$  is  {\em almost-doubling} when there exists a non increasing function $\theta:(0,1]\to\R^+\setminus\{0\}$ such that :
\begin{itemize}
\item
$\theta (1)=0$ and $\lim_{r\to 0^+}\frac{ \theta(r)}{\log(r)}=0$ 
\item
there is  a constant $C\geq 1$ such that for all $x\in\zu^d$ and $r\in (0,1]$ one has 
\begin{equation}
\label{ad}
C^{-1}e^{-\theta(r)}\mu(B(x,r))\le \mu(B(x,2r)) \le Ce^{\theta(r)}\mu(B(x,r)).
\end{equation}
\end{itemize}

 When $\theta\equiv 0$, then $\mu$ is said to be {\em doubling}.

Doubling and almost-doubling measures occupy a special place in geometric measure theory since they are easier to deal with in many situations - such properties guarantee a certain stability of the values of $\mu$ in the sense that $\mu(B)$ and $\mu(B')$   have comparable values as soon as $B$ and $B'$ are two balls of comparable radii that are close to each other. It is thus important to investigate the possible combination of these properties with the multifractal ones, as  done in the following theorem proved in \cite{BS-FP}.

\begin{theorem}
\label{th-prescrmu2}
Let $\tau:\R\to \R$ be concave, non-decreasing, with   $ - d = \tau(0^+)\leq \tau(1) =0$.    

Then there exists an HM almost doubling measure   $\mu\in \mathcal{M}(\zu^d)$  with full support in $\zu^d$ such that $\tau_\mu=\tau$ and $\mu$ satisfies the multifractal formalism, i.e. $D_\mu=\tau^*  $.
\end{theorem}

Although   Gibbs measures associated with H\"older regular potentials and smooth maps provide examples of doubling measures with non-trivial multifractal behavior, it may seem surprising  that the almost doubling property (which, as said above, limits the local variations of a measure) does not constitute  a constraint from the multifractal formalism standpoint:    every (admissible) concave mapping can be obtained as  the singularity spectrum of a compactly supported probability measure satisfying the multifractal formalism.

Theorem \ref{th-prescrmu2}   leaves open interesting questions in ergodic theory and dynamical systems, and geometric measure theory, which to the best of our knowledge are not completely addressed yet: 
\begin{enumerate}
\item
Can the almost doubling property be simplified in a "simple" doubling property in Theorem \ref{th-prescrmu2}?
\item
Given an almost doubling measure $\mu$, is there a doubling measure $\widetilde \mu$ with same multifractal behavior as $\mu$?
\item
 Is it possible to find a H\"older potential on a suitable  dynamical system such that the associated invariant measure satisfies the multifractal formalism with a  $L^q$-spectrum given in advance? 
\end{enumerate}

\begin{remark}
\label{rk11}
In Theorem \ref{th-prescrmu2}, it is possible to impose additional conditions on the  measures $\mu$ so that the same result ($D_\mu = \tau^*$) holds. One useful condition, which will be used later,  is the following.

\begin{definition}\label{mildcond1} Let $\Theta$ be the set of non decreasing functions $\theta:\N\to\R^*_+$ such that:
\begin{enumerate}
\item $\theta(j)=o(j)$ as $j\to\infty$
\item $\theta(0)=0$
\item
 for all $\ep>0$, there exists $j_\ep\in\N$ such that for all  $j'\ge j\ge j_\ep$,
$
\theta(j')-\theta(j)\le \ep(j'-j)$.
\end{enumerate}
A  measure   $\mu\in \mathcal{M}(\zu^d)$ (or $\mu\in \mathcal{M}(\R^d)$) satisfies property (P) if there exist  $C,s_1, s_2>0$ such that:

\begin{itemize}
\item[(P1)] for all $j\in\N$ and $\lambda\in \mathcal D_j$, if $\mu(\la)\neq 0$, then 
\begin{equation}
\label{minmaj1}
C^{-1} 2^{-j s_2}\le \mu(\lambda)\le C 2^{-j s_1}.
\end{equation}

\item[(P2)] There exists $\theta\in\Theta$ such that for all $j,j'\in\N$ with $j'\ge j$,  and  all $\lambda, \widetilde \lambda\in\mathcal D_j$  such that $\mu(\la)\neq 0$, $\mu(\tilde \la)\neq 0$, $\partial\lambda\cap \partial  \widetilde \lambda\neq\emptyset$,  and $\lambda'\in\mathcal D_{j'}$ such that $\lambda'\subset \lambda$:
\begin{equation}\label{propmu}
C^{-1}2^{-\theta(j)} 2^{(j'-j)s_1}  \mu(\lambda')\le \mu(\widetilde \lambda)\le C2^{\theta(j)} 2^{(j'-j)s_2}  \mu(\lambda').
\end{equation}
\end{itemize} 
\end{definition}
Heuristically, this last condition yields for every dyadic cube $\la \in \mathcal{D}_j $ a control of the $\mu$-mass of the cubes $\tilde\la  \in \mathcal{D}_{\tilde j} $ with $\tilde j\geq j$ and  $\tilde\la \subset 3\la$. It is  easily  checked on self-similar measures satisfying an open-set condition for instance. 

In \cite{BS-FP}, it is proved that there exist measures satisfying (P) for which the conclusion of   
Theorem \ref{th-prescrmu2} holds.
\end{remark}
 
\subsection{Prescription of multifractal formalism for functions}
\label{sec-prescrfunc}

While the definition of the $L^q$-spectrum for measures is quite standard and intuitive, finding a suitable formula for the  $L^q$-spectrum of functions  is not straightforward. Indeed, one easily sees that equation \eqref{defzeta} does not allow one to catch and describe the local regularity characteristics of smooth functions (with pointwise exponents greater than 1). Many alternative formulas have been proposed, and most of them are based on wavelets. 
It is thus useful at this point to set the notation concerning wavelets coefficients and wavelet leaders. 

\medskip

Let $\Phi:\R^d\to \R$ be a scaling function and consider an associated  family of  smooth   wavelets $\Psi=\{\psi^{(i)}\}_{i=1,...,2^d-1}$ belonging to $C^r(\R^d)$, with $r\in\N^*$ (for a general construction, see~\cite[Ch. 3]{Meyer_operateur}).  For simplicity, we   assume that $\Phi$ and the wavelets $\Psi$  are compactly supported \cite{Daub92}. For every $j\in \Z$, recall that   $\mathcal{D}_j $ is the set of dyadic cubes of generation $j$, i.e. if $k=(k_1,..., k_d)\in \Z^d$ and 
$$\lambda_{j,k} := \prod_{i=1,...,d} [k_i2^{-j}, (k_i+1)2^{-j}) \subset \R^d$$
 then
$\mathcal{D}_j =\{ \lambda_{j,k}  : k\in \Z^d\}$ (see the beginning of Section \ref{sec-prescrmu}). Further we consider the set
 $$\Lambda_j =\{\lambda=(i,j,k) :  \,k\in \Z^d, \ i\in \{1,...,2^d-1\} \},$$
and      $\Lambda=\bigcup_{j\in\Z} \Lambda_j $. By abuse of notation, $\lambda\in \Lambda_j$ will still be called a dyadic cube of generation $j$ and identified with $\lambda = \lambda_{j,k}\in \mathcal D_j$.

For every cube $\lambda=(i,j,k)\in\Lambda$, we denote by $\psi_\lambda$ the function $x\mapsto \psi^{(i)}(2^j x-k)$. The set of  functions $2^{dj/2}\psi_\lambda$, $j\in\Z$, $\lambda\in\Lambda_j$, forms a Hilbert basis of $L^2(\R^d)$, so that  every  $f\in L^2(\R^d)$  can be expanded as
$$ 
f=\sum_{j\in\Z}\sum_{\lambda\in\Lambda} d_\lambda\psi_\lambda,  \  \mbox{ with } \   
d_\lambda=\int_{\R^d} 2^{dj}\psi_\lambda(x) f(x)\,{d}x,
$$
where equality holds in $L^2$ (we will work with smooth functions, so equality will also hold pointwise). Observe that we choose an $L^\infty$ normalization for the so-called {\em wavelet coefficients}   $(d_\lambda)_{\lambda\in\Lambda} $  of $f\in  L^2(\R^d)$ (more generally, of $f\in L^p(\R^d)$ for some $p\in [1,\infty]$). For $f\in L^2(\R^d)$, define also   for $ k\in\Z^d$
\begin{equation}
\label{defbeta}
\beta(k)=\int_{\R^d} f(x) \Phi (x-k)\, {d}x.
\end{equation}

Finally, for   a function $f\in L^p(\R^d)$ with $p\in [1,\infty]$ whose wavelet coefficients are denoted by $(d_\lambda)_{\lambda\in \Lambda}$, the wavelet leader associated with $\lambda\in \mathcal{D}_j $ is  
$$
d^L_\lambda=\sup_{\lambda'\in \Lambda, \,\lambda'\subset 3\lambda}|d_{\lambda'}|,
$$ 
where for $\lambda\in \mathcal{D}_j$, $3\lambda$ stands for the cube with same center as $\lambda$ and radius $\frac{3}{2}2^{-j}$ (it is the cube that contains $\lambda $ as well as its $2^d-1$ neighbors in $\mathcal{D}_j$).  While wavelet coefficients are usually sparse (only a few coefficients carry the important information about $f$),  wavelet leaders   possess a strong hierarchical structure    since $0\leq d^L_{\lambda'}\leq d^L_\lambda$ when $\lambda'\subset \lambda$. 

\begin{remark}
Although the notations for wavelet coefficients and wavelet leaders do not mention the function $f$, they highly depend on $f$ and we should never forget about it!  
\end{remark}

Wavelet coefficients and wavelet leaders characterize  the pointwise H\"older exponents: indeed, if  $f\in C^\epsilon(\R^d)$ for some $\epsilon>0$, then  for every $x_0\in \zu^d$ one has 
\begin{equation}
\label{eq-carac}
h_f(x_0)=\liminf_{j\to\infty} \frac{\log d^L_{\lambda_{j}(x_0)}}{\log(2^{-j})},
\end{equation}
where $\lambda_{j}(x_0)$ is the unique cube $\lambda\in \mathcal{D}_j$ that contains $x_0$  (see \cite{JAFF_WAVTECH}).

It was quite  clear from the beginning  that a formula based on increments like \eqref{defzeta} was not stable neither mathematically nor numerically. To circumvent this difficulty, the  idea of introducing wavelets  (whose computation requires   local means,     bringing simultaneously a  numerical stability crucial for applications and a natural connection with  characterizations of standard functional spaces, see Section \ref{sec-typical})  was introduced by Alain Arn\'eodo and his collaborators. Two formulations are nowadays recognized to be the most relevant:

\begin{itemize}
   
 \item Formula based on wavelets:
\begin{equation}
\label{def-eta}
T_f (q,j)=  \sum_{\lambda\in \Lambda_j :  d_\la\neq 0} | d_{\lambda}|^q   \ \longrightarrow \ \eta_f(q) =\liminf_{j\to +\infty} \frac{\log_2 T_f (q,j)}{ {-j}}.
\end{equation}

 \item Formula based on  wavelet leaders:
\begin{equation}
\label{def-L}
L_f (q,j)=  \sum_{\lambda\in \mathcal{D}_j  : d_\la^L\neq 0}  | d^L_{\lambda}|^q   \ \longrightarrow \ L_f(q) =\liminf_{j\to +\infty}  \frac{\log_2 L_f (q,j)} { {-j}}.
\end{equation}
\end{itemize}

Even if wavelets brought some stability in the computations, wavelet leaders are now recognized as the most efficient, relevant and numerically exploitable measurements of local and global regularity. In particular, the hierarchical structure of wavelet leaders (i.e. $0\leq d_\lambda \leq d_{\lambda'}$ as soon as $\lambda \subset \lambda'$) makes all computations easier and more stable \cite{AJW-1}.

\begin{definition} 
The wavelet \ml formalism WMF ({\em resp.} wavelet leader \ml formalism WLMF) is satisfied for a function $f$  on an interval $J\subset \R^+$ when $D_f(H) = (\eta_f  )^*(H) $  ({\em resp.} $D_f(H)= (L_f )^*(H)$) for every $H\in J$.

We also say that a function $f$ satisfies the weak wavelet leader \ml formalism   (weak-WLMF)  on an interval $J\subset \R^+$  when there exists an increasing sequence $(j_n)_{n\geq 1}$  of integers such that if $\widetilde L_f(q) =\liminf_{n\to +\infty}  \frac{\log_2 L_f (q,j_n)} { {-j_n}}$,  then $D_f (H) = (\widetilde L_f)^*(H)$ for every $H\in J$.
\end{definition}

\begin{remark}
The above definition of formalisms depends {\em a priori} on the chosen wavelets $\Psi$. Actually it does not depend on $\Psi$ in the increasing part of the multifractal spectrum \cite{JAFF_WAVTECH}, but it does in the decreasing part. For simplicity, we do not mention this dependence in the notations.
\end{remark}

Let $\mathcal{S}_d$ be the set of admissible singularity spectra for functions satisfying a \ml formalism, i.e.
\begin{equation}
\label{def-Sd}
\mathcal{S}_d= \left\{ \sigma:\ \R^+ \to [0,d]\cup\{-\infty\}: \! \begin{cases}  \  \sigma \mbox { is compactly supported in $(0,+\infty)$, concave, \!\!\! } \!\!  \!\!\!  \\ \mbox{ with maximum equal to $d$}. \end{cases}  \!\!\! \right\}.
\end{equation}
\medskip

We are now able to state  the result on multifractal formalism prescription for functions. 
\begin{theorem}

For every   mapping $\sigma \in \mathcal{S}_d$, there exists a function $f\in L^2(\R^d)$ satisfying the WLMF and whose singularity spectrum is equal to $\sigma$.
\end{theorem}
 
 \begin{proof}

Observe that if a function $f$ has its wavelet coefficients $d_\lambda$   given by $\mu(\lambda)$ for some probability  measure $\mu\in \mathcal{M}(\zu^d)$, then for every choice of  $\alpha, \beta>0$, the function $f_{\alpha,\beta}$ whose wavelet coefficients are $\tilde d_\lambda:= d_\lambda^\alpha2^{-j\beta}$ satisfies 
$$\mbox{ for every $H\geq 0$, } \ \ D_{f_{\alpha,\beta}} (H) = D_f\left( \frac{H-\beta}{\alpha}\right).$$
This simply follows from \eqref{eq-carac} and the fact that $h_{f^{\alpha,\beta}}(x_0) =  \alpha h_f(x_0)+\beta$ for all $x_0$.

\smallskip

Let $\sigma:\R^+\to [0,d]\cup \{-\infty\} \in \mathcal{S}_d$  be a mapping satisfying the conditions to be a singularity spectrum of a function satisfying a \ml formalism.

Let $\alpha,\beta$ be two strictly positive real numbers such that the mapping $\sigma_{\alpha,\beta}(H) = \sigma(\alpha H +\beta)$ satisfies $\sigma_{\alpha,\beta}(H)\leq H$ and there exists $H_0>0$ such that $\sigma_{\alpha,\beta}(H_0)=H _0$. The existence  of ($\alpha,\beta)$ is an exercise (notice that    ($\alpha,\beta)$ need not be unique).

Theorem \ref{th-prescrmu2} provides us with a measure $\mu$ satisfying the \ml formalism for measures and $D_\mu = \sigma_{\alpha,\beta}$. 

Then,  Theorem \ref{th-Fmu}  yields that the function $F_\mu$ whose wavelet coefficients are given by $d_\lambda=\mu(\lambda)$ has the same singularity spectrum as $\mu$. In addition, comparing \eqref{def-tau} with   \eqref{def-L}, and using the hierarchical structure of the measure (i.e. $\mu(\lambda') \leq \mu(\lambda)$ whenever $\lambda'\subset \lambda$), one sees that $\tau_\mu(q) = L_{F_\mu}(q)$ for every $q\in \R$, hence $F_\mu$ satisfies the WLMF.

Finally,  using the first remark of this proof, the function $F$ whose wavelet coefficients   equal $\mu(\lambda)^\alpha2^{-j\beta}$ has its  singularity spectrum equal to $\sigma$ and satisfies the WLMF.
 \end{proof}

We thus have a complete answer for the prescription of \ml formalism for functions.
But at this point, one may have the feeling that the functions we built are mathematical toy examples. The purpose of the last sections is to explain that for any choice of  concave  admissible mapping $\sigma$, there are natural functional spaces in which typical functions have exactly $\sigma $ as singularity spectrum.  This  confirms and strengthens the overall presence of multifractals in most of science fields, and reinforces the position of \ml machinery as legitimate tool  in signal processing and data analysis.

\section{Typical multifractal behavior in classical functional spaces}
\label{sec-typical}

As emphasized above, it is possible to find mathematical models that mimic large classes of  \ml  behavior, in particular including all concave singularity spectra. This last part of the   results is key, since for real-life data (multi-dimensional and/or multivariate signals, images, ...) only estimates for the $L^q$-spectrum are numerically accessible (based on log-log plots on a well-chosen range of scales). Indeed, the standard paradigm is   to assume that the discrete data $f$ (say, a signal) is obtained from discrete samples of a mathematical  model obeying a multifractal formalism, and to consider that the Legendre transform of the estimated $L^q$-spectrum contains relevant information regarding the distribution of the singularities of  $f$ (somehow extrapolating on Frisch-Parisi heuristics). This Legendre transform is thus viewed as an "approximation" of the singularity spectrum of the data, although the   meaning of the singularity spectrum of a discretized signal is not made precise. The obtained estimated singularity spectrum of the data $f$ possesses various characteristics (values of the largest and the smallest exponents, locations of the maximum, curvature of the concave spectrum at its maximum,...) which are then used as  classification tools between  numerous samples of a physical, medical,... phenomenon. This has proven to be relevant in various fields going from medicine (heart-beat rate and X-ray analysis) and turbulence \cite{LashermesRouxJaffardAbry} to, recently, more surprising areas (paintings analysis \cite{AWJ-2}, text analysis \cite{Leonarduzzi-text}).

\medskip

Inspired by these applications,  it is thus key to investigate whether the mathematical objects we regularly meet satisfy a multifractal formalism (so that all these heuristics described above lie on   solid mathematical grounds). In this survey, we focus on "typical" objects in the sense of Baire: in a Baire space $E$, a property $\mathcal{P} $ of elements $x\in E$  is {\em typical} or {\em generic}  when the set $\{x\in E: x$ satisfies $\mathcal{P}\}$ is a residual set, i.e. its complement is included in a first Baire category set (a  union of countably many nowhere dense sets in $E$).

\medskip

Regularity properties of typical  functions  have been explored since the pioneer works of Banach \cite{Banach-typical} or Mazurkiewicz \cite{Mazu1} for instance. 
The seminal  result concerning \ml properties  of typical functions  is due to Buczolich and Nagy, who proved the following \cite{BUC_Nagy}.

\begin{theorem}
\label{th-bucnagy}
Let $\mathcal{M}on(\zu)$ be the set of continuous monotone functions $f:\zu\to\R$ equipped with the supremum norm of functions. Typical functions  in  $\mathcal{M}on(\zu)$  are multifractal with singularity spectrum equal to $D_f(H) = H \cdot {\bf 1\!\!\! 1}_{[0,1]}(H) + (-\infty)\cdot {\bf 1\!\!\! 1}_{(1,+\infty]}(H)$.\end{theorem}

Theorem \ref{th-bucnagy} was the starting point of an abundant literature on the subject, examples of which are given in the following. The method consists  first in finding  an upper bound for the singularity spectrum of all functions in $\mathcal{M}on(\zu)$ (here, the diagonal $\sigma(H)=H)$), then an explicit function $F_{typ}$ whose local behavior is the one suspected to be typical, and finally  to construct a countable sequence $(A_n)_{n\geq 1}$ of sets of functions,    dense in $\mathcal{M}on(\zu)$, which are for a given $n$, really close to  $F_{typ}$ at a given scale (depending on $n$). If the parameters are correctly settled, the intersection of the  $(A_n)_{n\geq 1}$ will be the   set of typical functions with \ml behavior similar to that of $F_{typ}$. 

The proof is  based on a careful analysis on local oscillations of functions, and simultaneous constructions of Cantor-like sets $E_f(H)$ carrying the sets of points with pointwise H\"older exponent equal to $H$, for every $f\in \bigcap_{n\geq 1} A_n$.

\medskip

After Theorem \ref{th-bucnagy}, the first direction    consisted in exploring the typical behavior in other standard functional spaces.
The first, spectacular,  results were obtained by Jaffard \cite{JAFF_FRISCH}, who implemented  the same strategy as \cite{BUC_Nagy} but added  wavelet tools to deal with the important examples of H\"older and Besov spaces.

\begin{theorem}
\label{th-besov}
1) Let $\alpha>0$ and consider the space  $C^\alpha(\zu^d)$ of $\alpha$-H\"older functions on $\zu^d$. Typical functions in $C^\alpha(\zu^d)$ are monofractal and satisfy 
$$D_f(H)= d\cdot {\bf 1\!\!\! 1}_{\{\alpha\}}(H) + (-\infty)\cdot {\bf 1\!\!\! 1}_{[0,+\infty]\setminus\{\alpha\}}(H).$$

\medskip

2) Let $p\geq 1$ and $s>d/p$, and consider the Besov space $B^{s,p}_q(\zu^d)$. Typical functions in $B^{s,p}_q(\zu^d)$ are \ml and satisfy 
$$D_f(H)= p(H - (s-d/p) )\cdot {\bf 1\!\!\! 1}_{[ s-d/p,s]}(H) + (-\infty)\cdot {\bf 1\!\!\! 1}_{[0,+\infty]\setminus[ s-d/p,s]}(H).$$
In addition, typical functions satisfy the WLMF.

\end{theorem}

See figure \ref{figure-besov} for an illustration. 

\begin{figure}
\includegraphics[width=80mm,height=50mm]{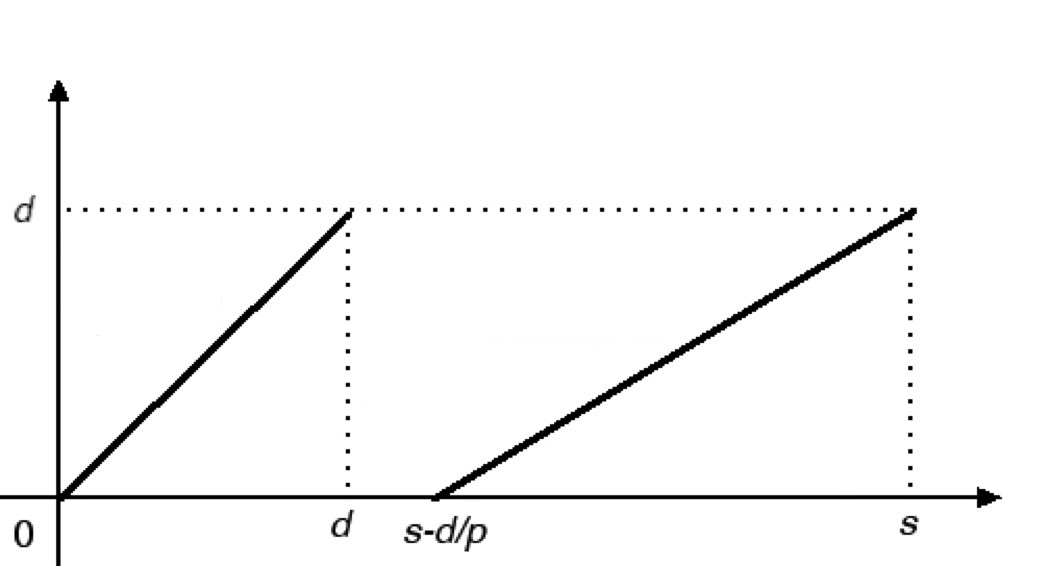}
\caption{Typical singularity spectra of  measures supported on $\zu^d$ ({\bf Left}) and of functions in $B^{s,p}_q(\R^d)$ ({\bf Right}).}
\label{figure-besov}
\end{figure}

Theorems \ref{th-bucnagy} and \ref{th-besov} are striking since they underline the preeminence  of multifractal properties for "everyday" functions. Jaffard also described the \ml behavior of typical functions belonging to countable intersections of Besov spaces, leading to a first answer to the Frisch-Parisi Conjecture. Although these results were a giant step in the domain, only increasing singularity spectra with restricted shapes can be obtained and the typical functions do not obey a satisfactory \ml formalism. Let us also mention that Besov spaces with    indices $s<d/p$ were also considered in \cite{JAFF_FRISCH}.
 
\medskip

Other directions have been investigated. The most natural one concerns probability measures: typical \ml properties were explored  in \cite {BuS2} for measures supported on $\zu^d$ and these results were extended by Bayart \cite{Bay} for measures supported on general compact sets.

\begin{theorem}
Let $K\subset\R^d$ be a compact set, and let $\mathcal{M}(K)$ be  the set of probability measures on $K$.

A typical measure  $\mu \in \mathcal{M}(K)$ satisfies for any $H\in [0, \dim(K))$, $ D_\mu(H)  =H $.

In addition, when the $\dim(K)$-Hausdorff measure of $K$ is strictly positive, then typical measures satisfy $D_\mu(\dim(K))  =\dim(K) $ and obey the multifractal formalism.

\end{theorem}

Another extension of typical monotone  functions is provided by  the set of monotone increasing in several variables:  A function $f : [0,1]^d \to \R$  is continuous monotone increasing in several variables (in short: MISV) if for all $i \in \{1, ..., d\}$, the coordinate functions
$$
 f^{(i)}(t) = f(x_1, ..., x_{i - 1}, t, x_{i+1}, ..., x_d)
 $$
  are continuous monotone increasing. The set of MISV functions is denoted by MISV$^d$.

With Z. Buczolich, we also investigated      the set $ {{\CC}}$ of continuous convex functions $f: {[0,1]^d}\to {\ensuremath {\mathbb R}}$.

Equipped with the supremum norm $\|\cdot\|$,  $ {\CC}$  and MISV$^d$ are  separable complete metric spaces.   In \cite{BuS6} and \cite{BuS4} we obtained the following results. 
\begin{theorem}
\label{th_extension}
 1) Typical functions in MISV$^d$  satisfy  
$$
D_f(H) =   (   d-1+H )\cdot {\bf 1\!\!\! 1}_{[ 0,1]}(H)+    (-\infty)\cdot {\bf 1\!\!\! 1}_{[0,+\infty]\setminus[ 0,1]}(H).
$$

2) Typical functions $f\in \CC$    satisfy  
$$
D_f(H) =  (   d-1 )  \cdot {\bf 1\!\!\! 1}_{\{0\}}(H)  + (   d-2+H)  \cdot {\bf 1\!\!\! 1}_{[ 1,2]}(H) + (-\infty)\cdot {\bf 1\!\!\! 1}_{[0,+\infty]\setminus[ 1,2]\cup\{0\}}(H).$$
 
\end{theorem}
 
\begin{figure}
  \begin{center}
  \includegraphics[width=9.0cm,height=7.4cm]{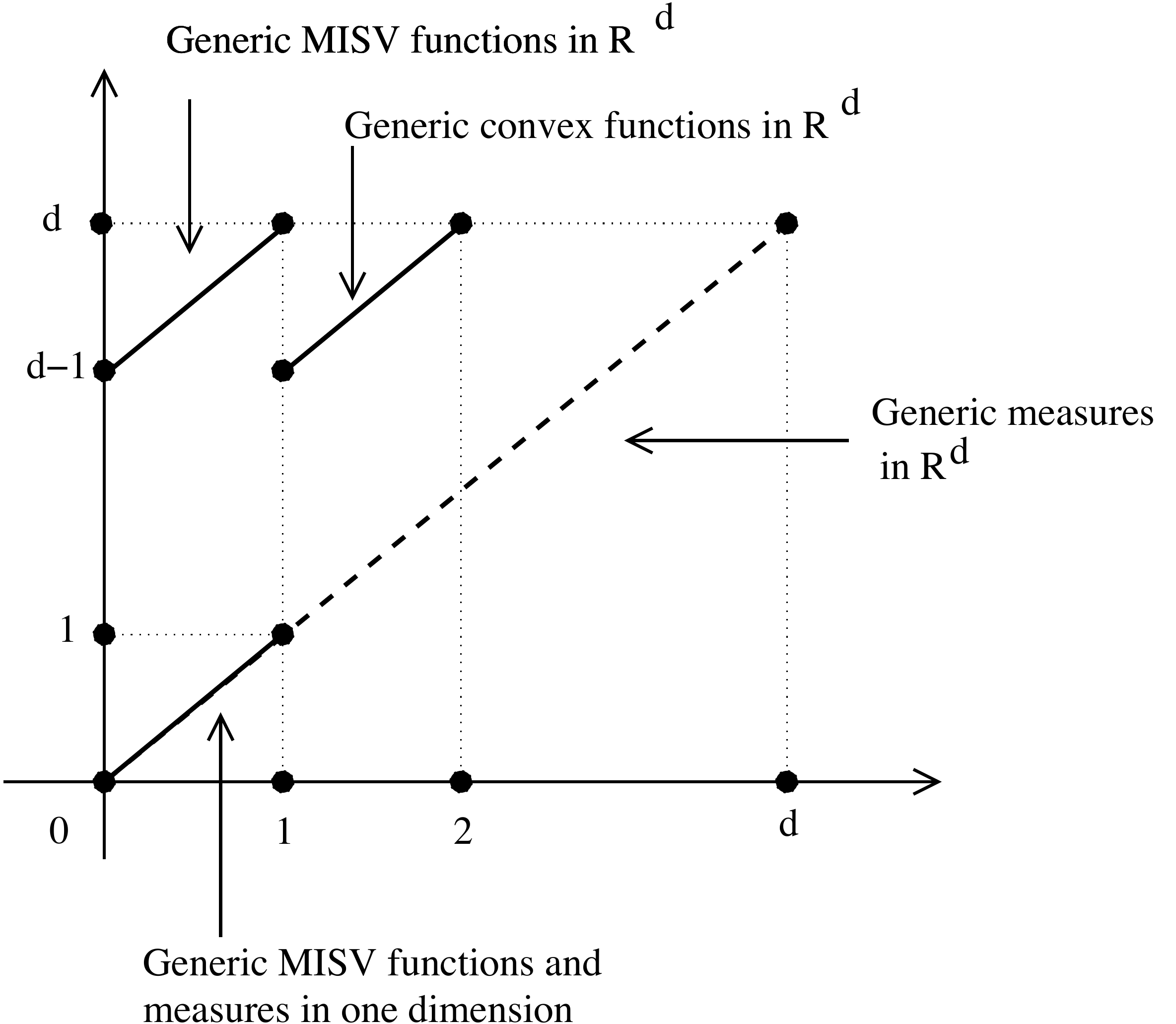}
  
    \caption{Typical singularity spectra for measures,   MISV and   convex functions}
     \label{fig0}
  \end{center}
\end{figure}
See Figure   \ref{fig0} for a comparison between   typical \ml behavior  in various functional spaces. This shall also be compared to Figure \ref{figureFP}.
 
It  appears clearly that in all  the previous situations,  the singularity spectra of typical functions have the same shape: it is an affine, increasing, mapping, with no decreasing part. 

Other functional spaces, called   $S_\nu$ spaces  were built in   \cite{Aubry-snu}, in which typical functions all exhibit a singularity spectrum which is {\em visibly increasing} in the sense of  \cite{MamanS}, enlarging the class of possible typical \ml behavior in functional spaces. In addition, these typical functions do not satisfy a \ml formalism in the sense of Definition \ref{def-formalism}.

In order to break this limitation (no   decreasing part in the singularity spectrum), new (and natural) functional spaces have been introduced in \cite{BS-FP}.

\section{Besov spaces in multifractal environment}
\label{sec_besovmultifractal}

 Since standard functional spaces do not fulfill our requirements (i.e. typical functions in such spaces do not exhibit concave singularity spectra), it is natural to ask whether there are other functional spaces in which typical functions have any singularity spectrum given in advance, and satisfy a \ml formalism. This solves the Frisch-Paris conjecture as stated in Conjecture \ref{conjfp}.

\medskip

Let $\mathcal{B}(\R^d)$ be the Borel sets included in $\R^d$, and let us introduce the  set of  H\"older set functions
\begin{equation}
\label{defCd}
\mathcal{C}(\R^d) := \left\{\mu:\mathcal{B}(\R^d) \to \R^+  \mbox{ such that }\ \begin{cases}  \exists \,  s_1, s_2\geq 0, \ \forall \, I\subset \R ^d \mbox{ with } |I|\leq 1, \\  \ \   |I|^{s_2} \leq   \mu(I) \leq  |I|^{s_1} \end{cases}\right\}.
\end{equation}
For $\mu\in \mathcal{C}(\R^d)$ and $s\in \R$, we write 
\begin{align*}
\mu^{s} (I) & = \mu(I) ^{s},\\
\mu^{(s)} (I) & = \mu(I) |I|^{s}.
\end{align*}
 
We will use the following notation: for $x,y\in \R^d$, $B[x,y]$ is the smallest Euclidean ball that contains $x$ and $y$.

\begin{definition}
Let $h\in \R^d$, $f:\R^d \to\R $, and consider   the finite difference operator $\Delta_h f: x\mapsto  f(x+h)-f(x)$. Define for $n\geq 2$ by iteration $\Delta ^n_hf := \Delta_h ( \Delta^{n-1}_h f)$.

For every  set function $\mu\in \mathcal{C}(\R^d)$, let us introduce for $n\geq 2$ 
\begin{equation}
\label{defdeltat}
\Delta ^{\mu,n}_hf(x)  = \frac{\Delta ^n_hf(x)}{\mu(B[x,x+nh])}.
\end{equation}

The  {\em $\mu$-adapted $n$-th order modulus of continuity of $f$} on  $\R^d$ is defined for $t>0$  by
\begin{align}
\label{defomegat}
  \omegat_n(f,t ) _p &= \sup_{t/2\leq |h|\leq t} \| \Delta ^{\mu,n}_hf \|_{L^p(\R^d)}.
\end{align}
  \end{definition}

It is trivial to check that  that when $\mu(I)=1$ for every set $I$, then $ \omegat_n(f,t) _p$ coincides with the so-called homogeneous  $n$-th order modulus of continuity of $f$  
$$
 \omega_n(f,t) _p = \sup_{t/2\leq |h|\leq t} \| \Delta ^n_hf \|_{L^p(\R^d)}.
$$

\begin{definition}
\label{defbmu}
Let $\mu\in \mathcal{C}(\R ^d)$ associated with exponents $0<s_1\leq s_2$  in \eqref{defCd}.

Let $n\geq s_2$. For $1\leq p,q\leq +\infty$, the Besov space  in $\mu$-environment  $B^{\mu,p}_q(\R^d)$ is the space of those functions $f:\R^d\to\R$  such that  $ \|f\|_{L^p(\R^d)}  <+\infty$ and 
\begin{equation}
\label{defbpqmuosc}
 |f|_{B^{\mu ,p}_q} = \|2^{jd/p}  (\omegat_n(f, 2^{-j}) _p)_{j\geq 1}\|_{\ell^q(\N)} <+\infty.
\end{equation}
Finally, let us introduce the spaces 
\begin{equation}
\label{defbmu2}
\btilde^{\mu,p}_q(\R^d) = \bigcap_{0<\ep<s_1/2} B^{\mu^{(-\ep)},p}_q(\R^d).
\end{equation}
\end{definition}

The reader can check  that ${B}^{\mu,p}_{q}(\R^d)$, when endowed with the topology induced by the norm $\|f\|_{B^{\mu p}_q} =\|(\beta(k))_{k\in \Z^d}\|_p+|f|_{B^{\mu,p}_q}$, forms  a  Banach space (recall \eqref{defbeta} for the definition of $\beta(k)$).

The intuition behind Definition \ref{defbmu} consists in introducing some space-dependent constraints that will create heterogeneity at all scales. Indeed,  when a function $f$ belongs to $B^{\mu,p}_q(\R^d)$,  its oscillations $\Delta ^n_hf(x)$ must be very small in certain regions (around points $x$ where $\mu(B(x,r)) \sim r^\alpha$ with $\alpha$ large), while in   other regions   (where $\mu(B(x,r)) \sim r^\alpha$ with $\alpha$ small) the control of the oscillations can be relaxed.  

\mk

In \cite{BS-FP}, a wavelet characterization of  ${B}^{\mu,p}_{q}(\R^d)$ and $\btilde^{\mu,p}_q(\R^d) $ is proved when $\mu$  is an almost-doubling set function satisfying condition (P) (recall equation \eqref{ad} and Remark \ref{rk11}). Observe indeed that Definition \ref{mildcond1} of the condition (P) for measures can easily be extended for set functions $\mu\in \mathcal{C}(\R^d)$.

For this, let us introduce a second semi-norm for $f\in L^p(\R^d)$ : we set
$$|f|_{p,q,\mu} = \|(A_j)_{l\geq 1}\|_ {\ell^q(\N)},  \mbox{ where  } A_j=  \left(\sum_{\lambda\in \Lambda_j} \left|\frac{d_\lambda}{\mu(\lambda)}\right|^p\right)^{1/p} .$$
The following inequalities are proved in \cite{BS-FP}. \begin{theorem}
\label{th-waveletcarac}
Let $\mu\in \mathcal{C}(\R^d)$ be an almost doubling set function satisfying condition (P), and let $\Phi$ be a scaling function associated with   wavelets  $\Psi$  (see   Section \ref{sec-prescrfunc}).

Let $p\geq 1$, and $q\in [1,+\infty]$.

  Assume  that   the wavelets $\Psi$  are compactly supported,  belong to the standard Besov space $B^{s,p}_{q}(\R^d)$ for some $s>d/p +s_2$, and  possess  at least $\lfloor s_2 \rfloor+1$   vanishing moments ($s_1$ and $s_2$ are the exponents associated with $\mu$ in \eqref{defCd}).
 
For every $0<\ep <s_1$, there exists a constant $C>1$ (not depending on $f$)  such that  
\begin{align}
\label{eqnorm}
\|f\|_{L^p}+ |f|_{{B}^{\mu,p}_{q}} & \leq C( \|f\|_{L^p}+|f|_{\mu^{(+\ep)},p,q} )\\
\label{eqnorm2}
 \|f\|_{L^p}+|f|_{\mu,p,q}& \leq C(  \|f\|_{L^p}+ |f|_{{B}^{\mu^{(+\ep)},p}_{q}  }).
\end{align}
 
 Moreover, when $\mu$ is doubling, \eqref{eqnorm}  and \eqref{eqnorm2}  hold for $\ep= 0$, and the norms $\|f\|_{L^p}+ |f|_{p,q,\mu} $ and $ \|f\|_{L^p}+|f|_{B^{\mu,p}_q} $ are equivalent.
%
\end{theorem}

Last theorem supports the idea that $\btilde^{\mu,p}_q(\R^d) $  is the right space to work with, since it is characterized by wavelet coefficients, while the spaces  $B^{\mu,p}_q(\R^d) $ are not (unless $\mu$ is doubling).

\mk

The main theorem in \cite{BS-FP}  is the following.
\begin{theorem}
\label{main} 
Let $\sigma\in \mathcal{S}_d$ be an admissible singularity spectrum (recall \eqref{def-Sd}). Call $H_s$ the smallest value at which $\sigma(H)=d$.

There exists   an almost doubling set function $\mu\in \mathcal{C}(\R^d)$ satisfying condition (P)  and $p\in [ 1,+\infty]$ such that for every $q\in [1,+\infty]$, typical functions $f \in  \widetilde {B}^{\mu,p}_{q}(\R^d)$ possess the following properties:
\begin{itemize}
\item
$D_f = \sigma$
\item
$f$ satisfies the WLMF for every $H\leq H_s$.

\item
$f$ satisfies the weak-WLMF for every $H> H_s$.

\end{itemize}

In addition:

\begin{itemize}
\item when $p=+\infty$, typical functions in   $  \widetilde {B}^{\mu,p}_{q}(\R^d)$   satisfy $D_f=D_\mu$. 

\item
when $\mu$ is doubling, the same holds for $  {B}^{\mu,p}_{q}(\R^d)$   instead of $  \widetilde {B}^{\mu,p}_{q}(\R^d)$.
\end{itemize}
\end{theorem}

Also, given $\sigma\in \mathcal{S}_d$,  from the proof in \cite{BS-FP} it can be checked that the couple $(\mu,p)$ in Theorem \ref{main} is not unique. 

Theorem \ref{main} brings a solution to the Frisch-Parisi conjecture (Conjecture \ref{conjfp}). The fact that the (strong) \ml formalism holds only for  the increasing part of the singularity spectrum (when $H\leq H_s$) seems to be unavoidable. A heuristic   explanation of   the weak validity of the \ml formalism in the decreasing part of the spectrum (and not the full validity)  is that    functions have usually sparse wavelet representations,  generating very large values for negative values of $q$ for $L_f(q,j)$ on some values of $j$.

\medskip

Let us conclude this section by mentioning that a deeper study of the $B^{\mu,p}_q$ and $\widetilde B^{\mu,p}_q$ spaces is performed in \cite{BS-FP}, leading to results that have their own interest. More precisely,  a uniform upper bound for the singularity spectrum of all functions in $B^{\mu,p}_q$ and $\widetilde B^{\mu,p}_q$ is found, as well as the singularity spectrum of typical functions in these spaces for large classes of almost-doubling measures $\mu$. Without giving  details on the results, it appears that the   singularity spectra $D_f$ of typical functions $f$ may have very different shapes depending on the initial measure $\mu$, and the proofs involve many arguments coming from geometric measure theory, ergodic theory and harmonic analysis.

\section{Perspectives}
\label{sec_perspectives}

First of all, we are far from being exclusive on generic dimensional results in analysis (see for instance  \cite{Fraysse-jaff-kah,  Gruslys-1}), and many other   regularity properties shall definitely be studied from the Baire genericity standpoint.

In this survey we focused on the notion of Baire genericity -  the same issues can (and must) be addressed in the {\em prevalence} sense. Many results regarding prevalent \ml properties have been obtained, see  \cite{AuMaSeu,fraysse-jaffard-1, Fraysse1,Ol1,Ol2} amongst many references, and asking whether prevalent  properties coincide with generic ones can sometimes bring some surprises (when they do not coincide).

Finally, one challenging research  direction consists in establishing \ml properties for (clas\-ses of) solutions to ordinary or partial differential equations, as well as for the stochastic counterparts. Indeed, \ml ideas originate from the study of turbulence and other physical phenomena that are ruled by ODEs,  SDEs or (S)PDEs,  and it would be a fair return to demonstrate the multifractality of (some of) those functions that are solutions to such equations.
A few examples already exist (i.e., Burgers equation with a Brownian motion as initial condition \cite{bertoin-jaffard} and  large classes of stochastic jump diffusions \cite{BFJS-markov,yang2018ihp}), but they are only a first step toward a systematic \ml  analysis of solutions to (partial) differential  equations, which will certainly require the development of new techniques and approaches.

\section*{Acknowlegments}

The author wish to thank the organizers of the fantastic conference "Fractal Geometry and Stochastics 6", and the referee for all her/his relevant   and useful comments on this text.


\bibliographystyle{plain}
\bibliography{Biblio_stephane}

\begin{thebibliography}{10}

\bibitem{AJW-1}
P.~Abry, S.~Jaffard, and H.~Wendt.
\newblock Irregularities and scaling in signal and image processing:
  Multifractal analysis.
\newblock In M.~Frame Ed., editor, {\em Benoit Mandelbrot: A Life in Many
  Dimensions}, pages 31--116. World Scientific, 2015.

\bibitem{AWJ-2}
P.~Abry, H.~Wendt, and S.~Jaffard.
\newblock When {V}an {G}ogh meets {M}andelbrot: Multifractal classification of
  painting's texture.
\newblock {\em Signal Processing}, 93(3):554--572, mars 2013.

\bibitem{Aubry-snu}
J.-M. Aubry, S.~Bastin, F.~Dispa, and S.~Jaffard.
\newblock The spaces $s_\nu$: new spaces defined with wavelet coefficients and
  related to multifractal analysis.
\newblock {\em Int. J. Appl. Math. Statist.}, 7:82--95, Feb. 2007.

\bibitem{AuMaSeu}
J.-M. Aubry, D.~Maman, and S.~Seuret.
\newblock Local behavior of traces of besov functions: Prevalent results.
\newblock {\em J. Func. Anal.}, 264(3):631--660, 2013.

\bibitem{AJT-prescription}
A.~Ayache, S.~Jaffard, and Murad~S. Taqqu.
\newblock Wavelet construction of generalized multifractional processes.
\newblock {\em Rev. Mat. Iberoamericana}, 23(1):327--370, 2007.

\bibitem{Banach-typical}
S.~Banach.
\newblock {\"U}ber die {B}aire'sche kategorie gewisser funktionenmengen.
\newblock {\em Studia Math.}, 3:174-- 179, 1931.

\bibitem{Barralinverse}
J.~Barral.
\newblock Inverse problems in multifractal analysis of measures.
\newblock {\em Ann. Ec. Norm. Sup}, 48(6):1457--1510, 2015.

\bibitem{BFJS-markov}
J.~Barral, N.~Fournier, S.~Jaffard, and S.~Seuret.
\newblock A pure jump {M}arkov process with a random singularity spectrum.
\newblock {\em Ann. Probab.}, 38(5):1924--1946, 2010.

\bibitem{BSFmu}
J.~Barral and S.~Seuret.
\newblock From multifractal measures to multifractal wavelet series.
\newblock {\em J. Fourier Anal. Appl.}, 11(5):589--614, 2005.

\bibitem{BS-FP}
J.~Barral and S.~Seuret.
\newblock Besov spaces in multifractal environement, and the {F}risch-{P}arisi
  conjecture.
\newblock {\em preprint}, 2019.

\bibitem{Bay}
F.~Bayart.
\newblock Multifractal spectra of typical and prevalent measures.
\newblock {\em Nonlinearity}, 26:353--367, 2013.

\bibitem{bertoin-jaffard}
J.~Bertoin and S.~Jaffard.
\newblock Solutions multifractales de l'{\'e}quation de burgers.
\newblock {\em Matapli}, 52:19--28, 1997.

\bibitem{BroMichPey}
G.~Brown, G.~Michon, and J.~Peyri\`ere.
\newblock On the multifractal analysis of measures.
\newblock {\em J. Stat. Phys.}, 66:775--790, 1992.

\bibitem{BUC_Nagy}
Z.~Buczolich and J.~Nagy.
\newblock {H}\"older spectrum of typical monotone continuous functions.
\newblock {\em Real Anal. Exchange}, pages 133--156, 1999.

\bibitem{BuS2}
Z.~Buczolich and S.~Seuret.
\newblock Typical borel measures on $[0,1]^d$ satisfy a multifractal formalism.
\newblock {\em Nonlinearity.}, 23(11):7--13, 2010.

\bibitem{BuS6}
Z.~Buczolich and S.~Seuret.
\newblock {H}\"older spectrum of functions monotone in several variables.
\newblock {\em J. Math. Anal. Appl.}, 1:110--126, 2011.

\bibitem{BuS3}
Z.~Buczolich and S.~Seuret.
\newblock Measures and functions with prescribed singularity spectrum.
\newblock {\em J. Fractal Geometry}, 1(3):295--333, 2014.

\bibitem{BuS4}
Z.~Buczolich and S.~Seuret.
\newblock Multifractal properties of typical convex functions.
\newblock {\em Monatshefte f{\"u}r Mathematik}, 2019.

\bibitem{Daub92}
I.~Daubechies.
\newblock {\em Ten Lectures on Wavelets}.
\newblock SIAM, 1992.

\bibitem{Fraysse1}
A.~Fraysse.
\newblock Regularity criteria of almost every function in a {S}obolev space.
\newblock {\em J. Func. Anal.}, pages 1806--1821, 2010.

\bibitem{fraysse-jaffard-1}
A.~Fraysse and S.~Jaffard.
\newblock How smooth is almost every function in a {S}obolev space?
\newblock {\em Rev. Mat. Iberoamericana}, 22:663--682, 2006, 22, pp.663-682.

\bibitem{Fraysse-jaff-kah}
A.~Fraysse, S.~Jaffard, and J.-P. Kahane.
\newblock Quelques propri{\'e}t{\'e}s g{\'e}n{\'e}riques en analyse.
\newblock {\em Notes aux CRAS, S{\'e}rie I, Math.}, 340:645--651, 2005.

\bibitem{FrischParisi}
U.~Frisch and D.~Parisi.
\newblock Fully developed turbulence and intermittency in turbulence, and
  predictability in geophysical fluid dynamics and climate dynamics.
\newblock {\em International school of Physics ``Enrico Fermi", course 88,
  edited by M.~Ghil, North Holland}, pages 84--88, 1985.

\bibitem{Gruslys-1}
V.~Gruslys, V.~Jonusas, V.~Mijovi{\'c}, O.~Ng, L.~Olsen, and I.~Petrykiewicz.
\newblock Dimensions of prevalent continuous functions.
\newblock {\em Monatshefte f{\"u}r Mathematik}, 166(2):153--180, 2012.

\bibitem{Halsey}
T.C. Halsey, M.H. Jensen, L.P. Kadanoff, I.~Procaccia, and B.I. Shrai\-man.
\newblock Fractal measures and their singularities: the characterisation of
  strange sets.
\newblock {\em Phys.\ Rev.~A}, 33:1141--1151, 1986.

\bibitem{Hentschel}
H.G. Hentschel and I.~Procaccia.
\newblock The infinite number of generalized dimensions of fractals and strange
  attractors.
\newblock {\em Physica D}, pages 435--444, 1983.

\bibitem{JAFFNOTE}
S.~Jaffard.
\newblock Exposants de {H}\"older en des points donn\'es et coefficients
  d'ondelettes.
\newblock {\em C.~R.~Acad.\ Sci.\ Paris S\'erie I}, 308:79--81, 1989.

\bibitem{Jaff-prescription}
S.~Jaffard.
\newblock Construction de fonctions multifractales ayant un spectre de
  singularit{\'e}s prescrit.
\newblock {\em C.R.A.S.}, 315(1):19--24, 1992.

\bibitem{jaffard-prescribed}
S.~Jaffard.
\newblock Functions with prescribed {H}\"older exponent.
\newblock {\em Appl. Comput. Harmon. Anal.}, 2:400--401, 1995.

\bibitem{JAFF_FRISCH}
S.~Jaffard.
\newblock On the {F}risch-{P}arisi conjecture.
\newblock {\em J. Math. Pures Appl.}, 79(6):525--552, 2000.

\bibitem{JAFF_WAVTECH}
S.~Jaffard.
\newblock Wavelet techniques in multifractal analysis.
\newblock In {\em Fractal Geometry and Applications: A Jubilee of Benoit
  Mandelbrot}, Proc. Symposia in Pure Mathematic. AMS Providence, RI, 2004.

\bibitem{LashermesRouxJaffardAbry}
B.~Lashermes, S.~Roux, S.~Jaffard, and P.~Abry.
\newblock Comprehensive multifractal analysis of turbulent velocity using
  wavelet leaders.
\newblock {\em Eur. Phys. J. B.}, 61(2):201--215, 2008.

\bibitem{Leonarduzzi-text}
R.~Leonarduzzi, P~Abry, S.~Jaffard, H.~Wendt, L.~Gournay, T.~Kyriacopoulou,
  C.~Martineau, and C.~Martinez.
\newblock $p$-leader multifractal analysis for text type identification.
\newblock In {\em IEEE Int. Conf. Acoust., Speech, and Signal Proces.
  (ICASSP)}, New Orleans, USA, march 2017.

\bibitem{MamanS}
D.~Maman and S.~Seuret.
\newblock Fixed points for the multifractal spectrum map.
\newblock {\em Constructive approximation}, 43(3):337--356, 2016.

\bibitem{M2}
B.~B. Mandelbrot.
\newblock Multiplications al\'{e}atoires it\'{e}r\'{e}es et distributions
  invariantes par moyennes pond\'{e}r\'{e}es.
\newblock {\em C. R. Acad. Sci. Paris}, 278:289--292 et 355--358, 1974.

\bibitem{Mazu1}
S.~Mazurkiewicz.
\newblock Sur les fonctions non d{\'e}rivables.
\newblock {\em Studia Math.}, 3:92--94, 1931.

\bibitem{Meyer_operateur}
Y.~Meyer.
\newblock {\em Ondelettes et op{\'e}rateurs I}.
\newblock Hermann, 1990.

\bibitem{Muzy-arneodo-1}
J.-F. Muzy, E.~Bacry, and A.~Arn\'eodo.
\newblock Multifractal formalism for fractal signals: The structure-function
  approach versus the wavelet-transform modulus-maxima method.
\newblock {\em Phys. Rev. E}, 47:875--884, 1993.

\bibitem{Olsen}
L.~Olsen.
\newblock A multifractal formalism.
\newblock {\em Adv. Math.}, 116:92--195, 1995.

\bibitem{Ol1}
L.~Olsen.
\newblock Fractal and multifractal dimensions of prevalent measures.
\newblock {\em Indiana Univ. Math. J.}, 59 (2):661--690, 2010.

\bibitem{Ol2}
L.~Olsen.
\newblock Prevalent ${L}^q$-dimensions of measures.
\newblock {\em Math. Proc. Cambridge Philos. Soc}, 149(3):553--571, 2010.

\bibitem{yang2018ihp}
X.~Yang.
\newblock Multifractality of jump diffusion processes.
\newblock {\em Ann. Inst. Henri Poincar{\'e} Probab. Stat.}, 54(4):2042--2074,
  2018.

\end{thebibliography}

\end{document}